\documentclass[twoside, 11pt]{article}

\usepackage[a4paper, left = 1in, right = 1in, textheight = 22cm]{geometry}

\usepackage{enumerate}

\usepackage{amsmath, amsfonts, amssymb, mathtools, amsthm}
\usepackage{dsfont}

\usepackage{hyperref}
\usepackage{cleveref}

\usepackage{biblatex}
\DeclarePairedDelimiter\abs{\lvert}{\rvert}
\DeclareMathOperator{\Span}{span}

\DeclareMathOperator{\cl}{cl}
\newcommand{\dis}{\displaystyle}
\renewcommand{\leq}{\leqslant}
\renewcommand{\geq}{\geqslant}

\newcommand{\C}{\mathbb{C}}
\newcommand{\F}{\mathbb{F}}
\newcommand{\N}{\mathbb{N}}
\renewcommand{\P}{\mathbb{P}}

\newcommand{\e}{\varepsilon}

\theoremstyle{plain}
\newtheorem{theorem}{Theorem}[section]
\newtheorem{proposition}[theorem]{Proposition}
\newtheorem{lemma}[theorem]{Lemma}

\theoremstyle{definition}
\newtheorem{definition}[theorem]{Definition}

\crefname{proposition}{Proposition}{Propositions}
\crefname{lemma}{Lemma}{Lemmata}
\crefname{equation}{equation}{equations}
\crefname{theorem}{Theorem}{Theorems}

\makeatletter
\let\oldabs\abs
\def\abs{\@ifstar{\oldabs}{\oldabs*}}

\title{\textbf{Double algebraic genericity of universal harmonic functions on trees in the general case}}
\author{C. A. Konidas, V. Nestoridis}
\date{\vspace{-5ex}}

\begin{document}
\pagestyle{myheadings}
\markboth{Double algebraic genericity of universal harmonic functions on trees in the general case}{C. A. Konidas, V. Nestoridis}
\maketitle
\begin{abstract}
\noindent It has been shown that the set of universal functions on trees contains a linear subspace except zero, dense in the space of harmonic functions.
In this paper we show that the set of universal functions contains two linear subspaces except zero, dense in the space of harmonic functions that intersect only at zero.
We work in the most general case that has been studied so far, letting our functions take values over a topological vector space.
\end{abstract}
{\em AMS classification numbers}: 05C05, 60J45, 60J50, 30K99, 46M99\smallskip

\section{Introduction}
Let $T$ be the set of vertices of a rooted tree with root $x_0$.
For each $n \in \N$ we define $T_n$ to be the set of all vertices at distance $n$ from $x_0$ and we also define $T_0 = \{x_0\}$.
Given a vertex $x \in T$, we write $S(x)$ for the set of the children of x.
We assume that for all $n \in \N$ the set $T_n$ is finite and that for every $x \in T$ the set $S(x)$ has at least two elements.
Since $T = \bigcup_{n=0}^{\infty}{T_n}$ and each $T_n$ is finite and nonempty, the set $T$ is infinite denumerable.
With each $x \in T$ and $y \in S(x)$ we associate a real number $q(x,y) > 0$, which we think of as the probability of transition from vertex $x$ to vertex $y$, such that
$$\sum_{y \in S(x)}{q(x,y)} = 1$$
for all $x \in T$.

We define the boundary of $T$, denoted by $\partial T$, as the set of all infinite geodesics originating from $x_0$.
More specifically, an element $e$ of $T$ is of the form $e = \{z_n \in T: n \in \N\}$, where $z_1 = x_0$ and $z_{n+1} \in S(z_n)$ for all $n \in \N$.
For each $x \in T$ we define the boundary sector of $x$ to be the set $B_x = \{ e\in \partial T: x \in e\}$.

Given $x \in T\setminus \{x_0\}$, it must be that $x \in T_n$ for some $n \in \N$.
We set $$p(B_x) = \prod_{j=0}^{n-1}{q(y_j,y_{j+1})},$$ where $y_0 = x_0, y_n = x$ and $y_{j+1} \in S(y_j)$ for all $j \in \{1,\dots,n-1\}$.
For each $n \in \N$ we consider $\mathcal{M}_n$ to be the $\sigma$-algerba on $\partial T$ generated by $\{B_x: x  \in T_n\}$.
We can extend $p$ to a probability measure $\P_n$ on the measurable space $(\partial T, \mathcal{M}_n)$.
One can check that $\mathcal{M}_n \subseteq \mathcal{M}_{n+1}$ and $\P_{n+1}|_{\mathcal{M}_n} = \P_n$ for all $n \in \N$.
By Kolmogorov's Consistency Theorem there exists a complete probability measure $\P$ on the measurable space $(\partial T, \mathcal{M})$, where $\mathcal{M}$ is the completion of the $\sigma$-algebra on $\partial T$ generated by the set $\bigcup_{n=1}^{\infty}{\mathcal{M}_n}$, such that $\P|_{\mathcal{M}_n} = \P_n$ for all $n \in \N$.

Let $E$ be a separable topological vector space metrizable by a translation invariant metric, over a field $\F$.
For every $x \in T$ and $y \in S(x)$ we fix $w(x,y) \in \F\setminus\{0\}$ such that for all $x \in T$
$$\sum_{y \in S(x)}{w(x,y)} = 1.$$

A function $f \in E^T$ is called generalized harmonic if and only if
$$f(x) = \sum_{y \in S(x)}{w(x,y)f(y)}$$
for all $x \in T$.
The set of generalized harmonic functions is denoted by $H(T,E)$.
From this point forward we drop the world generalized and refer to the elements of $H(T,E)$ simply as harmonic functions.

A function $\psi : \partial T \to E$ is said to be $\mathcal{M}$-measurable (respectively $\mathcal{M}_n$-measurable) if and only if $\psi^{-1}(V) \in \mathcal{M}$ (respectively $\psi^{-1}(V) \in \mathcal{M}_n$) for all open $V \subseteq E$.
We obtain the space $L^0(\partial T, E)$ by identifying $\P$-almost everywhere equal $\mathcal{M}$-measurable functions and we equip it with the topology of convergence in probability.

For every $f \in H(T,E)$ and $n \in \N$ we define the function $\omega_n(f): \partial T \to E$ by $\omega_n(f)(e) = f(z)$, where $z$ is the unique element of $e \cap T_n$ and $e \in \partial T$.
Notice that $\omega_n(f)$ is constant on each $B_x$ for $x \in T_n$ with the family $(B_x)_{x\in T_n}$ being a partition of $\partial T$, thus $\omega_n(f)$ is $\mathcal{M}_n$-measurable and moreover $\mathcal{M}$-measurable.

A generalized harmonic function $f \in H(T,E)$ is called universal if and only if the sequence $(\omega_n(f))_{n=1}^{\infty}$ is dense in $L^0(\partial T, E)$. The set of universal functions is denoted by $U(T,E)$. The following theorem was proven in \cite{BIEHLER2}.

\begin{theorem}
The set $U(T,E) \cup \{0\}$ contains a vector space which is dense in $H(T,E)$, that is we have algebraic genericity for the set $U(T,E)$.
\end{theorem}

In this paper we improve this result by proving the following theorem.

\begin{theorem}
\label{thm:wanted}
There exist two vector spaces $F_1,F_2$ contained in $U(T,E)\cup\{0\}$ which are dense in $H(T,E)$ such that $F_1 \cap F_2 = \{0\}$, that is we have double algebraic genericity for the set $U(T,E)$.
\end{theorem}

A weaker version of this theorem was proven in \cite{KONIDAS}, where $E$ was assumed to be a normed space over $\C$.
In this paper we generalise that proof for $E$ a topological vector space over a field $\F$, satisfying of course some technical assumptions which are stated at length in the sequel.

We begin by demonstrating that $L^0(\partial T, E)$ is a topological vector space over $\F$.
Then, in order to prove our result, we consider two disjoint subsets of $U(T,E)$ which have been studied before.
The first is the set of frequently universal harmonic functions $U_{FM}(T,E)$ (see \cref{def:UFM}) for which algebraic genericity has been proven in \cite{BIEHLER2} with the same assumptions as the ones we are considering for $E$.
The second is the set denoted as $X(T,E)$ (see \cref{def:X}) for which algebraic genericity has only been proven for  $E$ a normed space over $\C$, in \cite{KONIDAS}.

Our aim is to prove algebraic genericity for the set $X(T,E)$ in the general case.
The main idea is to consider $E^\N$, which is also a topological vector space over $\F$, and the set $X(T,E^\N)$.
We prove that there exists a function $f = (f_1,f_2, \dots) \in X(T,E^\N)$ such that the sequence $(f_n)_{n=1}^{\infty}$ is dense in $H(T,E)$, and that the span of $f_1,f_2,\dots$ lies in $X(T,E)$.

More information about universal harmonic functions on trees can be found in the articles \cite{ABAKUMOV1, ABAKUMOV2}.
Investigations in this paper are more closely related to results found in the articles \cite{BIEHLER1,BIEHLER2} and in the preprint \cite{KONIDAS}.
The concept of algebraic genericity is further discussed in \cite{LINEARITY}.

\section{Preliminaries}

Let $\F$ be a field endowed with a metric such that addition, multiplication and the inverse function are continuous.
Let $E$ be a separable topological vector space over the field $\F$ that is metrizable by a translation invariant metric $d$.
We consider the space $L^0(\partial T, E)$, obtained by identifying $\P$-almost everywhere $\mathcal{M}$-measurable functions, and the metric
\begin{equation*}
\label{eq:P}
P(\psi, \phi) = \int_{\partial T}{\frac{d(\psi(e), \phi(e))}{1 + d(\psi(e), \phi(e))}} \, d\P(e) \tag{$\star$}
\end{equation*}
where $\psi,\phi \in L^0(\partial T, E)$.
The topology induced by the metric $P$ is that of convergence in probability.
Since $E$ is separable and by construction of the measure space $(\partial T, \mathcal{M}, \P)$, there exists a sequence $(\chi_n)_{n=1}^{\infty}$ dense in $L^0(\partial T, E)$ such that for all $n \in \N$ the function $\chi_n$ is $\mathcal{M}_{k(n)}$-measurable for some $k(n) \in \N$.

In general, the set of measurable functions from a measurable space to an arbitrary topological vector space needs not be a vector space.
In our case however the process of constructing the measurable space $(\partial T, \mathcal{M})$ yields the required structure.
The next proposition contains the main results about the structure of $L^0(\partial T, E)$.

\begin{proposition}
\label{prop:structure}
The following are true.

\begin{enumerate}[(i)]
\item The set $L^0(\partial T, E)$ considered with pointwise addition and $\F$-scalar multiplication is a vector space over $\F$.

\item Addition and scalar multiplication in $L^0(\partial T, E)$ are continuous with respect to convergence in probability.

\end{enumerate}
\end{proposition}

\begin{proof}\leavevmode

\begin{enumerate}[(i)]
\item It suffices to show that if $\psi:\partial T\to E$ and $\phi: \partial T \to E$ are two $\mathcal{M}$-measurable functions and $a \in \F$ is arbitrary, then the functions $\psi + \phi$ and $a\psi$ are $\mathcal{M}$-measurable.
Since the set $\{\chi_n : n \in \N\}$ is dense in $L^0(\partial T, E)$ we can find a sequence in that set that convergences to $\psi$ in probability.
We can then find a subsequence of that sequence that convergences to $\psi$ pointwise almost everywhere.
By reindexing we can find a sequence $(\psi_n)_{n=1}^{\infty}$ in the set $\{\chi_n : n \in \N\}$ such that $\psi_n \to \psi$ pointwise almost everywhere and each function $\psi_n$ is $\mathcal{M}_{\tau(n)}$-measurable for some $\tau(n) \in \N$.
By the same argument we can find a sequence $(\phi_n)_{n=1}^{\infty}$ in the set $\{\chi_n : n \in \N\}$ such that $\phi_n \to \phi$ pointwise almost everywhere and each function $\phi_n$ is $\mathcal{M}_{\ell(n)}$-measurable for some $\ell(n) \in \N$.
Thus there exists a set $A \in \mathcal{M}$ with $\P(A) = 1$ such that $\psi_n(e) \to \psi(e)$ and $\phi_n(e) \to \phi(e)$ for all $e \in A$.
Since addition in $E$ is continuous $\psi_n(e) + \phi_n(e) \to \psi(e) + \phi(e)$ for all $e \in A$.
By setting $r(n) = \max\{\tau(n), \ell(n)\}$ for all $n \in \N$ we have that $\mathcal{M}_{\tau(n)} \subseteq \mathcal{M}_{r(n)} $ and $\mathcal{M}_{\ell(n)} \subseteq \mathcal{M}_{r(n)}$ for all $n \in \N$.
Therefore, both $\psi_n$ and $\phi_n$ are $\mathcal{M}_{r(n)}$-measurable for all $n \in \N$ and so they are constant on all the sets $B_x$ for $x \in T_{r(n)}$.
This implies that the function $\psi_n + \phi_n$ is constant on all the sets $B_x$ for $x \in T_{r(n)}$.
Thus, the function $\psi_n +\phi_n$ is $\mathcal{M}_{r(n)}$-measurable and therefore $\mathcal{M}$-measurable.
So, we have that $\psi + \phi$ is an almost everywhere pointwise limit of the sequence of the $\mathcal{M}$-measurable functions $(\psi_n + \phi_n)_{n=1}^{\infty}$, taking values over the metric space $(E,d)$.
Also, the measure space $(\partial T, \mathcal{M}, \P)$ is complete.
These imply that the function $\psi+\phi$ is $\mathcal{M}$-measurable.
Proving that $a\psi$ is $\mathcal{M}$-measurable relies on scalar multiplication in $E$ being continuous, but the proof is similar and thus omitted.

\item We prove that scalar multiplication in $L^0(\partial T, E)$ is continuous.
To this end, let $a \in \F$ and $\psi \in L^0(\partial T, E)$ be arbitrary such that there exist a sequence $(a_n)_{n=1}^{\infty}$ in $\F$ with $a_n \to a$ and a sequence $(\psi_n)_{n=1}^{\infty}$ in $L^0(\partial T, E)$ with $\psi_n \to \psi$ in probability.
It suffices to show that $a_n\psi_n \to a\psi$ in probability.
Let $(a_{k_n}\psi_{k_n})_{n=1}^{\infty}$ be an arbitrary subsequence of $(a_n\psi_n)_{n=1}^{\infty}$.
Then $\psi_{k_n} \to \psi$ in probability as $(\psi_{k_n})_{n=1}^{\infty}$ is a subsequence of $(\psi_{n})_{n=1}^{\infty}$.
Thus there exists a further subsequence $(\psi_{k_{m_n}})_{n=1}^{\infty}$ of $(\psi_{k_n})_{n=1}^{\infty}$ such that $\psi_{k_{m_n}} \to \psi$ pointwise almost everywhere.
That is, there exists an $\mathcal{M}$-measurable set $A \subseteq \partial T$ with $\P(A) = 1$ such that for all $e \in A$ we have $\psi_{k_{m_n}}(e) \to \psi(e)$ in $E$.
Now, $(a_{k_{m_n}})_{n=1}^{\infty}$ is a subsequence of $(a_{n})_{n=1}^{\infty}$ and thus $a_{k_{m_n}} \to a$ in $\F$.
Scalar multiplication in $E$ is continuous, therefore for every $e \in A$ we have that $a_{k_{m_n}}\psi_{k_{m_n}}(e) \to a\psi(e)$ in $E$, that is, $a_{k_{m_n}}\psi_{k_{m_n}} \to a\psi$ pointwise almost everywhere.
In particular $a_{k_{m_n}}\psi_{k_{m_n}} \to a\psi$ in probability.
Convergence in probability in $L^0(\partial T, E)$ is induced by the metric $P$.
Thus the previous arguments prove that $a_n\psi_n \to a\psi$ in probability.
Proving that addition in $L^0(\partial T, E)$ is continuous relies on addition in $E$ being continuous, but the proof is similar and thus omitted.

\end{enumerate}
\end{proof}

We now consider the space $E^\N$ equipped with the product metric defined by the metric $d$ in each copy of $E$.
This way $E^\N$ becomes a separable topological vector space over $\F$, metrizable by the translation invariant product metric.
Thus we can define the set $L^0(\partial T, E^\N)$.

We identify each element $\psi \in L^0(\partial T, E^\N)$ with a sequence $(\psi_n)_{n=1}^{\infty}$ of functions $\psi_n: \partial T \to E$, by setting $\psi_n(e) = y_n$ for all $e \in \partial T$ when $\psi(e) = (y_n)_{n=1}^{\infty}$.
The next proposition examines the relation between $L^0(\partial T, E^\N)$ and $L^0(\partial T,E)^\N$.

\begin{proposition}
\label{prop:product}
The following are true.

\begin{enumerate}[(i)]

\item We have $L^0(\partial T, E^\N) = L^0(\partial T, E)^\N$. Equivalently, $\psi \in L^0(\partial T, E^\N)$ if and only if $\psi_n \in L^0(\partial T, E)$ for all $n \in \N$.

\item The topology of convergence in probability on $L^0(\partial T, E^\N)$ is the same as the product topology on $L^0(\partial T, E)^\N$, obtained when considering each copy of $L^0(\partial T, E)$ with the topology of convergence in probability.
Equivalently, for all sequences $(\psi^m)_{m=1}^{\infty}$ in $L^0(\partial T, E^\N)$ and all $\psi \in L^0(\partial T, E^\N)$, it is $\psi^m \to \psi$ in probability as $m \to \infty$ if and only if $\psi_n^m \to \psi_n$ in probability as $m \to \infty$ for all $n\in \N$.
\end{enumerate}
\end{proposition}

\begin{proof}\leavevmode

\begin{enumerate}[(i)]

\item Let us first assume that $\psi \in L^0(\partial T, E^\N)$.
We fix arbitrary $n \in \N$ and show that $\psi_n \in L^0(\partial T, E)$.
Let $V$ be an open subset of $E$.
If $$W = \dots \times E \times V \times E \times \dots$$ where $V$ is in the $n$-th position and every other set is $E$, then $W$ is an open subset of $E^\N$.
Since the function $\psi$ is $\mathcal{M}$-measurable $\psi^{-1}(W) \in \mathcal{M}$.
But $\psi^{-1}(W) = \psi_n^{-1}(V)$ and thus $\psi_n^{-1}(V) \in \mathcal{M}$.
Therefore the function $\psi_n$ belongs in $L^0(\partial T, E)$.
Conversely, let us assume that $\psi_n \in L^0(\partial T, E)$ for every $n \in \N$.
The topological vector space $E$ is separable and metrizable, thus it is second countable and so there exists a countable basis $\mathcal{C}$ of its topology.
Without loss of generality, we assume that $E \in \mathcal{C}$.
Then the set
$$\mathcal{B} = \left\{ \prod_{m \in \N}{V_m} : V_m \in \mathcal{C} \text{ and the set } \{m \in \N : V_m \neq E\} \text{ is finite} \right\}$$
is a countable basis for the topology of $E^\N$.
Thus the family $\mathcal{B}$ generates the Borel $\sigma$-algerba on $E^\N$ and so to prove that $\psi \in L^0(\partial T, E^\N)$, it suffices to show that $\psi^{-1}(W) \in \mathcal{M}$ for all $W \in \mathcal{B}$.
Let $W \in \mathcal{B}$.
Then there exists $n \in \N$ and $V_1,\dots,V_n \in \mathcal{C}$ such that
$$W = V_1 \times V_2 \times \dots \times V_n \times E \times E \times \dots.$$
We have
$$\psi^{-1}(W) = \psi_1^{-1}(V_1)\cap \psi_2^{-1}(V_2)\cap \dots \cap \psi_n^{-1}(V_n) \in \mathcal{M},$$
since the functions $\psi_1,\dots,\psi_n$ are $\mathcal{M}$-measurable.
Therefore $\psi \in L^0(\partial T, E^\N)$.

\item
We consider the metric $\widetilde{P}$ in $L^0(\partial T, E^\N)$ defined as
$$\widetilde{P}(\psi, \phi) = \int_{\partial T}{ \sum_{n=1}^{\infty}{\frac{1}{2^n}\frac{d(\psi_n(e),\phi_n(e))}{1 + d(\psi_n(e),\phi_n(e))} }} \, d\P(e) $$
for $\psi, \phi \in L^0(\partial T, E^\N)$.
Because the product metric in $E^\N$ is bounded by one, the topology induced by $\widetilde{P}$ is that of convergence in probability.
We notice that for every $\psi, \phi \in L^0(\partial T, E^\N)$, all $n \in \N$ and all $e \in \partial T$ it is
$$ \frac{1}{2^n}\frac{d(\psi_n(e),\phi_n(e))}{1 + d(\psi_n(e),\phi_n(e))} \geq 0.$$
Tonelli's theorem implies that
$$\widetilde{P}(\psi, \phi) = \sum_{n=1}^{\infty}{ \frac{1}{2^n} \int_{\partial T}{\frac{d(\psi_n(e),\phi_n(e))}{1 + d(\psi_n(e),\phi_n(e))}} \, d\P(e)}.$$
Therefore, recalling \cref{eq:P}, we have
$$\widetilde{P}(\psi, \phi) = \sum_{n=1}^{\infty}{\frac{P(\psi_n, \phi_n)}{2^n}}$$
for all $\psi,\phi \in L^0(\partial T, E^\N)$.
Now, let $(\psi^m)_{m=1}^{\infty}$ be a sequence in $L^0(\partial T, E^\N)$ and $\psi \in L^0(\partial T, E^\N)$.
Let us first assume that $\psi^m \to \psi$ in probability as $m \to \infty$, that is $\widetilde{P}(\psi^m , \psi) \to 0$ as $m \to \infty$.
Fix $n_0 \in \N$.
For all $m \in \N$ we have
$$0 \leq P(\psi^m_{n_0}, \psi_{n_0}) \leq 2^{n_0}\sum_{n=1}^{\infty}{\frac{P(\psi^m_{n}, \psi_{n})}{2^n}} = 2^{n_0}\widetilde{P}(\psi^m , \psi).$$
Therefore $P(\psi^m_{n_0}, \psi_{n_0}) \to 0 $ as $m \to \infty$, that is $\psi^m_{n_0} \to \psi_{n_0}$ in probability as $m \to \infty$.
Conversely, let us assume that $\psi^m_{n} \to \psi_n$ in probability as $m \to \infty$ for all $n \in \N$, that is $P(\psi^m_n, \psi_n) \to 0$ as $m \to \infty$ for all $n \in \N$.
For all $m \in \N$ and all $n \in \N$ we have $P(\psi^m_n, \psi_n) \leq 1$ and so
$$\frac{P(\psi^m_n, \psi_n)}{2^n} \leq \frac{1}{2^n} \text{~~~and~~~}
\sum_{n=1}^{\infty}{\frac{1}{2^n}} < \infty.$$
Lebesgue's dominated convergence theorem implies that
$$\lim_{m\to \infty}{ \sum_{n=1}^{\infty} { \frac{P(\psi^m_{n}, \psi_{n})}{2^n} } } =
\sum_{n=1}^{\infty} { \lim_{m\to \infty}{\frac{P(\psi^m_{n}, \psi_{n})}{2^n}}} = 0.$$
But for all $m \in \N$
$$\widetilde{P}(\psi^m, \psi) = \sum_{n=1}^{\infty} { \frac{P(\psi^m_{n}, \psi_{n})}{2^n} }.$$
Therefore $\widetilde{P}(\psi^m, \psi) \to 0$ as $m\to \infty$, which means that $\psi^m \to \psi$ in probability as $m \to \infty$.
\end{enumerate}
\end{proof}

Finally, we consider the topology of pointwise convergence in the space $E^T$ and its subset $H(T,E)$, which is induced by the metric
$$\rho(f,g) = \sum_{n=1}^{\infty}{\frac{1}{2^n}\frac{d(f(z_n), g(z_n))}{1 + d(f(z_n), g(z_n))}},$$
where $f,g \in E^T$ and $\{z_n:n \in \N\}$ is an enumeration of $T$.

\section{Double algebraic genericity}

Let $f \in H(T,E)$ be a harmonic function.
We have defined $f$ to be universal if and only if the sequence $(\omega_n(f))_{n=1}^{\infty}$ is dense in $L^0(\partial T, E)$.
This is equivalent to the set $\{n \in \N: \omega_n(f) \in V\}$ being nonempty for every nonempty open $V \subseteq L^0(\partial T, E)$.
As proven in \cite{BIEHLER2}, the space $L^0(\partial T, E)$ has no isolated points, and so $f \in U(T,E)$ if and only if the set $\{n \in \N: \omega_n(f) \in V\}$ is infinite for every nonempty $V \subseteq L^0(\partial T, E)$ open.

This observation motivates the following two definitions.

\begin{definition}
\label{def:UFM}
A harmonic function $f \in H(T,E)$ is said to be frequently universal if and only if for every nonempty open set $V \subseteq L^0(\partial T, E)$ the set $\{n \in \N: \omega_n(f) \in V\}$ has strictly positive lower density.
The set of frequently universal functions is denoted by $U_{FM}(T, E)$.
\end{definition}

\begin{definition}
\label{def:X}
A harmonic function $f \in H(T,E)$ is said to belong to the class $X(T,E)$ if and only if for every nonempty open set $V \subseteq L^0(\partial T, E)$ the set $\{n \in \N: \omega_n(f) \in V\}$ has upper density equal to one.
\end{definition}

Clearly $U_{FM}(T,E) \subseteq U(T,E)$ and $X(T,E) \subseteq U(T,E)$.
We now provide an alternative definition for the set $X(T,E)$.

\begin{proposition}
\label{prop:X}
A harmonic function $f \in H(T,E)$ belongs to the class $X(T,E)$ if and only if for every nonempty non-dense open set $V \subseteq L^0(\partial T, E)$ the set $\{n \in \N : \omega_n(f) \in V\}$ has lower density equal to zero.
\end{proposition}

\begin{proof}
Let us first assume that $f \in X(T,E)$.
Let $V$ be an arbitrary nonempty non-dense open subset of $L^0(\partial T, E)$.
Then $L^0(\partial T, E) \setminus \cl(V)$ is a nonempty open subset of $L^0(\partial T, E)$, where $\cl(V)$ is the closure of $V$ in $L^0(\partial T, E)$.
Because the function $f$ belongs to $X(T,E)$ the upper density of the set $\{n \in \N: \omega_n(f) \in L^0(\partial T,E)\setminus \cl(V) \}$ is equal to one and thus its complement $\{n \in \N: \omega_n(f) \in \cl(V) \}$ has lower density equal to zero.
Since the set $\{n \in \N: \omega_n(f) \in V\}$ is contained in the latter, we deduce that it has lower density equal to zero.
Conversely let us assume that $f \in H(T,E)$ is such that for every nonempty non-dense open set $V \subseteq L^0(\partial T, E)$ the set $\{n \in \N: \omega_n(f) \in V\}$ has lower density equal to zero.
Let $V$ be an arbitrary nonempty open subset of $L^0(\partial T, E)$.
The topology of convergence in probability in $L^0(\partial T, E)$ is metrizable and so there exists $F$ a closed subset of $L^0(\partial T, E)$ with nonempty interior such that $F \subseteq V$ and $F \neq L^0(\partial T, E)$.
Then $L^0(\partial T, E)\setminus F$ is a nonempty non-dense open subset of $L^0(\partial T,E)$.
By our assumption, the set $\{n \in \N: \omega_n(f) \in L^0(\partial T, E)\setminus F\}$ has lower density equal to zero.
Thus, its complement $\{ n \in \N: \omega_n(f) \in F\}$ has upper density equal to one.
The set $\{n \in \N: \omega_n(f) \in V\}$ contains the latter which implies that it has upper density equal to one.
Therefore, $f \in X(T,E)$.
\end{proof}

\cref{prop:X} and the existance of a nonempty non-dense open subset of $L^0(\partial T, E)$ lead to the proposition below, which has been proven differently in \cite{BIEHLER2}.

\begin{proposition}
\label{prop:disj}
The sets $U_{FM}(X,T)$ and $X(T,E)$ are disjoint, that is
$$U_{FM}(T, E) \cap X(T,E) = \varnothing.$$
\end{proposition}

As stated in the introduction, our aim is to prove algebraic genericity for the set $X(T,E)$, which combined with algebraic genericity for the set $U_{FM}(T,E)$, a known result, leads to our main result, that is double algebraic genericity for the set $U(T,E)$.

The proof of the algebraic genericity for the set $X(T,E)$ is split into two lemmata.
For the proof, we work in the space $L^0(\partial T, E^\N)$, which by \cref{prop:structure} is a topological vector space over $\F$.
Furthermore, one can see that $f = (f_1,f_2, \dots) \in H(T,E^\N)$ if and only if $f_n \in H(T,E)$ for every $n \in \N$ and if $f = (f_1,f_2,\dots) \in H(T,E)$, then $\omega_n(f) = (\omega_n(f_1), \omega_n(f_2), \dots)$ for all $n \in \N$.

We begin by stating and proving the first lemma.

\begin{lemma}
\label{lem:linearity}
If $f = (f_1,f_2,\dots) \in X(T,E^{\N})$, then the vector space $\Span\{f_n : n\in \N\}$ is contained in the set $X(T,E) \cup \{0\}$.
\end{lemma}

\begin{proof}
Let $a_1f_1 + \dots + a_sf_s$ be an arbitrary non-zero element of the set $\Span\{f_n : n\in \N\}$, where $s\in \N$ and $a_1,\dots,a_s \in \F$ with $a_s \neq 0$.
Let $V$ be an arbitrary nonempty open subset of $L^0(\partial T, E)$.
There exist some $\psi \in V$ and $\e > 0$ such that $B(\psi,\e) \subseteq V$.
For $i \in \{1,\dots, s\}$ we define
$$
b_i =
\begin{cases}
a_i, &\text{if } a_i \neq 0 \\
1, &\text{if } a_i = 0
\end{cases}.
$$
This way $b_1,\dots,b_s$ are all invertible elements of the field $\F$.
Let $\delta = \dis{\frac{\e}{s}}$ and consider the set
$$
\widehat{V} = b_1^{-1}B(0,\delta) \times b_2^{-1}B(0,\delta) \times \cdots \times b_s^{-1}B(\psi,\delta) \times L^0(\partial T, E) \times L^0(\partial T, E) \times \cdots.
$$
We now prove the inclusion
$$\{n \in \N: \omega_n(f) \in \widehat{V} \} \subseteq \{n \in \N: \omega_n(a_1f_1 + \dots + a_sf_s) \in V\}.$$
Let $n \in \N$ be arbitrary such that $\omega_n(f) \in \widehat{V}$.
Then for every $i \in \{1,\dots,s-1\}$ we have
$$\omega_n(f_i) \in b_i^{-1}B(0,\delta) \text{~~~and~~~} \omega_n(f_s) \in b_s^{-1}B(\psi,\delta).$$
By definition of $b_1,\dots,b_n$ and linearity of $\omega_n$ we conclude that for all $i \in \{1,\dots,s-1\}$
$$P(\omega_n(a_if_i), 0) < \delta \text{~~~and~~~} P(\omega_n(a_sf_s), \psi) < \delta,$$
where $P$ is defined by \cref{eq:P}.
As the metric $d$ is translation invariant, the metric $P$ is translation invariant and thus using the triangle inequality and the linearity of $\omega_n$ we obtain
\begin{align*}
P(\omega_n(a_1f_1 + \dots + a_sf_s), \psi) &\leq P(\omega_n(a_1f_1),0) + P(\omega_n(a_2f_2 + \dots + a_sf_s), \psi) \\
&\leq P(\omega_n(a_1f_1),0) + P(\omega_n(a_2f_2),0) + P(\omega_n(a_3f_3 + \dots + a_sf_s), \psi) \\
&~~\vdots \\
&\leq P(\omega_n(a_1f_1), 0) + \dots + P(\omega_n(a_{s-1}f_{s-1}),0) + P(\omega_n(a_sf_s),\psi) \\
&< s\cdot \delta.
\end{align*}
Since $s\cdot \delta = \e$, it follows that $\omega_n(a_1f_1+\dots+a_sf_s) \in B(\psi,\e)$, which concludes the proof of the desired inclusion, as $B(\psi, \e) \subseteq V$.
The components of the cartesian product that is the set $\widehat{V}$ are all open sets of $L^0(\partial T, E)$, as the latter is a topological vector space over $\F$ by \cref{prop:structure}.
Thus $\widehat{V}$ is a basic open subset of $L^0(\partial T,E)^\N$.
By \cref{prop:product}, $\widehat{V}$ is an open subset of $L^0(\partial T,E^\N)$ and it is also nonempty.
Since $f \in X(T,E^\N)$, the set $\{n \in \N: \omega_n(f) \in \widehat{V}\}$ has upper density equal to one, which, given the inclusion we proved, implies that the set $\{n \in \N: \omega_n(a_1f_1 + \dots + a_sf_s) \in V\}$ has upper density equal to one.
Therefore $a_1f_1 + \dots + a_sf_s $ belongs to the class $ X(T,E)$.
\end{proof}

We now state and prove the second lemma.

\begin{lemma}
\label{lem:density}
There exists some sequence $(f_n)_{n=1}^{\infty}$ of harmonic functions dense in $H(T,E)$ such that the function $f = (f_1,f_2,\dots) $ belongs to the class $X(T,E^\N)$.
\end{lemma}

\begin{proof}
The metric space $H(T,E)$ is separable.
Let $\{p_n: n \in \N\}$ be a dense sequence in $H(T,E)$.
The space $E^\N$ when considered with the product metric, which is translation invariant, is a separable topological vector space over $\F$.
Thus, as proven in \cite{BIEHLER2}, the class $X(T,E^\N)$ is dense in $H(T,E^\N)$ and in particular nonempty, so let $h = (h_1,h_2,\dots)$ be an element of $X(T,E^\N)$.
We now fix $n \in \N$.
There exists some $j_0(n)\in \N$ such that
$$\sum_{j= j_0(n) + 1}^{\infty}{\frac{1}{2^j}} < \frac{1}{n}.$$
There also exists some $N(n) \in \N$ such that
$$\{z_1,\dots,z_{j_0(n)} \} \subseteq \bigcup_{k=0}^{N(n)}{T_k},$$
where $\{z_j:j\in \N\}$ is the enumeration of $T$ through which the metric $\rho$ was defined.
We set $g_n(x) = p_n(x)-h_n(x)$ for all $x \in \bigcup_{k=0}^{N(n)}{T_k}$ and  $g_n(x) = g_n(x^-)$ for all $x \not\in \bigcup_{k=N(n)+1}^{\infty}{T_k}$, where $x^-$ is the father of $x$.
Notice that the function $g_n: T \to E$ defined this way is a generalized harmonic function and
$$\rho(p_n - h_n,g_n) = \sum_{j = 1}^{\infty}{\frac{1}{2^j}\frac{d((p_n - h_n)(z_j), g_n(z_j))}{1 + d((p_n - h_n)(z_j), g_n(z_j))}} \leq \sum_{j= j_0(n) + 1}^{\infty}{\frac{1}{2^j}} < \frac{1}{n}.$$
Moreover, if $k \in \N$ is such that $k > N(n)$ and $e \in \partial T$, then letting $x$ and $y$ be the unique elements of $e \cap T_{N(n)}$ and $e \cap T_k$ respectively $y$ is a descendant of $x$ and so $g_n(x) = g_n(y)$ by definition of $g_n$ for vertices at distance greater than $N(n)$ from the root of the tree.
That is, $\omega_k(g_n) = \omega_{N(n)}(g_n)$ for all $k > N(n)$.
We show that the sequence $(h_n + g_n)_{n=1}^{\infty}$ is as required.
Notice that the metric $\rho$ is translation invariant as the metric $d$ is translation invariant.
Thus $\rho(p_n, h_n + g_n) < 1/n$ for all $n \in \N$.
Since $H(T,E)$ has no isolated points and the sequence $(p_{n})_{n=1}^{\infty}$ is dense in $H(T,E)$, the sequence $(h_n + g_n)_{n=1}^{\infty}$ is dense in $H(T,E)$.
To complete the proof we show that the function $h+g = (h_1+g_1,h_2+g_2,\dots)\in H(T,E^\N)$ belongs to the class $X(T,E^\N)$.
It suffices to demonstrate that for every nonempty open basic set $V \subseteq L^0(\partial T,E)$ the set $\{n \in \N: \omega_n(h+g) \in V\}$ has upper density equal to one.
By \cref{prop:product}, a nonempty basic open set of $L^0(\partial T,E^\N)$ can be written as
$$V = V_1 \times \cdots \times V_s \times L^0(\partial T, E) \times L^0(\partial T, E) \times \cdots,$$
where $V_1,\dots,V_s$ are nonempty open subsets of $L^0(\partial T, E)$.
Let $L = \max\{N(1),\dots,N(s)\}$.
Then, for all $n \geq L$ and all $j \in \{1,\dots,s\}$ we have $\omega_n(g_j) = \omega_L(g_j)$.
We now prove the inclusion
$$\{n \geq L: \omega_n(h) \in V - \omega_L(g)\} \subseteq \{n\in\N: \omega_n(h+g) \in V\}.$$
Let $n \in\N$ with $n \geq L$ be such that $\omega_n(h) \in V - \omega_L(g)$.
Then $\omega_n(h_j) \in V_j - \omega_L(g_j)$ for all $j \in \{1,\dots,s\}$, which implies that $\omega_n(h_j)+ \omega_n(g_j) \in V_j$ for every $j \in \{1,\dots,s\}$.
But by linearity of $\omega_n$ it is $\omega_n(h_j + g_j) = \omega_n(h_j) + \omega_n(g_j)$ for all $j \in \{1,\dots,s\}$, hence $\omega_n(h_j+g_j) \in V_j$ for every $j \in \{1,\dots,s\}$, which in turn implies that $\omega_n(h+g) \in V$.
This concludes the proof of the desired inclusion.
Notice however that $V - \omega_L(g)$ is a nonempty open subset of $L^0(\partial T, E^\N)$, as $L^0(\partial T, E^\N)$ is a topological vector space over $\F$ by \cref{prop:structure}, and $h \in X(T,E^\N)$
Therefore the set $\{n \in \N: \omega_n(h) \in V - \omega_L(g)\}$ has upper density equal to one.
This implies that the set $\{n \geq L: \omega_n(h) \in V - \omega_L(g)\}$ has upper density equal to one and thus, given the inclusion we proved, the set $\{n\in\N: \omega_n(h+g) \in V\}$ has upper density equal to one.
\end{proof}

The next theorem follows by combining \cref{lem:density,,lem:linearity}.

\begin{theorem}
\label{thm:algX}
The set $X(T,E) \cup \{0\}$ contains a vector space which is dense in $H(T,E)$, that is, we have algebraic genericity for the set $X(T,E)$.
\end{theorem}

Finally, we state and prove the main theorem of this paper.

\begin{theorem}
The set $U(T,E) \cup \{0\}$ contains two vector spaces $F_1, F_2$ dense in $H(T,E)$ with $F_1 \cap F_2 = \{0\}$, that is, we have double algebraic genericity for the set $U(T,E)$.
\end{theorem}

\begin{proof}
By algebraic genericity of $U_{FM}(T,E)$, proven in \cite{BIEHLER2}, there exists a vector space $F_1$ dense in $H(T,E)$ such that $F_1 \subseteq U_{FM}(T,E) \cup \{0\}$.
By \cref{thm:algX} there exists a vector space $F_2$ dense in $H(T,E)$ such that $F_2 \subseteq X(T,E) \cup \{0\}$.
Since $U_{FM}(T,E)$ and $X(T,E)$ are contained in $U(T,E)$, the vector spaces $F_1,F_2$ are contained in $U(T,E)$.
Finally, by \cref{prop:disj}, it follows that $F_1 \cap F_2 \subseteq \{0\}$ and thus $F_1 \cap F_2 = \{0\}$.
\end{proof}

\printbibliography
\bigskip
\noindent
C. A. Konidas, V. Nestoridis\\
National and Kapodistrian University of Athens\\
Department of Mathematics\\
e-mail addresses: \href{mailto:xkonidas@gmail.com}{\tt xkonidas@gmail.com},
\href{mailto:vnestor@math.uoa.gr}{\tt vnestor@math.uoa.gr}

\end{document}